\documentclass[12pt]{article}
\usepackage{amsmath,amssymb,amsthm}
\newtheorem{theorem}{Theorem}
\newtheorem{lemma}{Lemma}

\renewcommand{\le}{\leqslant}
\renewcommand{\ge}{\geqslant}

\DeclareMathOperator{\cha}{char}

\begin{document}

\title{Normal form of
$m$-by-$n$-by-$2$ matrices for
equivalence}
\author{Genrich Belitskii%
\thanks{Partially supported by
Israel Science Foundation, Grant
216/98.}\\ Dept. of Mathematics,
Ben-Gurion University of the Negev,\\
Beer-Sheva 84105, Israel,
 genrich@cs.bgu.ac.il
 \and
Maxim Bershadsky,\\ Dept. of
Mathematics,
Ben-Gurion University of the Negev,\\
Beer-Sheva 84105, Israel,
maximb@math.bgu.ac.il
       \and
Vladimir V. Sergeichuk%
\thanks{
Partially supported by
FAPESP (S\~ao Paulo, Brazil),
processo 05/59407-6.}\\
Institute of Mathematics,
Tereshchenkivska 3, Kiev,
Ukraine\\sergeich@imath.kiev.ua}
\date{}

\maketitle

\begin{abstract}
We study $m\times n\times 2$
matrices up to equivalence and
give a canonical form of
$m\times 2\times 2$ matrices
over any field.


 \end{abstract}

\section{Introduction and the main
results}

Complex $2\times 2\times 2$
matrices up to equivalence were
classified by Schwartz
\cite{[8]} and Duschek
\cite{[2]}. Canonical forms of
complex and real $2\times
2\times 2$ matrices for
equivalence were given by
Oldenburger
\cite{[5]}–--\cite{[7]}; they
are presented in \cite[Section
IV, Theorem 1.1]{[9]}. Ehrenborg
\cite{[3]} also got a canonical
form of complex $2\times 2\times
2$ matrices for equivalence
basing on a collection of
covariants that separates the
canonical matrices.

In this paper we give a
canonical form of $m\times
2\times 2$ matrices for
equivalence over any field
$\mathbb F$, but first we
establish when $m\times n\times
2$ matrices, whose two $m\times
n\times 1$ submatrices are in
the Kronecker canonical form for
matrix pencils, are equivalent
over $\mathbb F$. Using an
alternative method, the authors
recently obtained in \cite{2a} a
canonical form of $m\times
2\times 2$ matrices for
equivalence over a field of
characteristic different from
$2$.

Note that the canonical form
problem for $m\times n\times 3$
matrices for equivalence is
wild; this means that it
contains the problem of
classifying pairs of linear
operators and therefore it
contains the problem of
classifying an arbitrary system
of linear operators (see, for
example, \cite[Theorems 4.5 and
2.1]{[1]}).

All matrices and spatial
matrices in this article are
considered over an arbitrary
field $\mathbb F$. By an
$m\times n\times q$
\emph{spatial matrix} over
$\mathbb F$ we mean an array
\begin{equation}\label{1.0}
{\cal A} =[a_{ijk}]_{i=1}
^m{}_{j=1}^n{}_{k=1}^q,\qquad
a_{ijk}\in\mathbb F.
\end{equation}
Two $m\times n\times q$ matrices
${\cal A} = [a_{ijk}]$ and
${\cal B} = [b_{ijk}]$ are
\emph{equivalent} if there exist
nonsingular $m\times m$,
$n\times n$, and $q\times q$
matrices
\begin{equation}\label{dno}
R=[r_{ii'}],\qquad
S=[s_{jj'}],\qquad T=[t_{kk'}]
\end{equation}
such that
\begin{equation}\label{1.1}
b_{i'j'k'}:=\sum_{ijk}
a_{ijk}r_{ii'}s_{jj'} t_{kk'}.
\end{equation}

This notion arises in the theory
of forms: each trilinear form
$f: U\times V\times W\to \mathbb
F$ on vector spaces with bases
$\{u_i\}_{i=1}^m$,
$\{v_j\}_{j=1}^n$, and
$\{w_k\}_{k=1}^q$ is given by
the spatial matrix \eqref{1.0}
with $a_{ijk}:=f(u_i,v_j,w_k)$.
Its entries change by
\eqref{1.1} if we go to other
bases with the transition
matrices \eqref{dno}.

We will give the spatial matrix
\eqref{1.0} by the $q$-tuple of
$m\times n$ matrices
\begin{equation*}\label{1.3}
{\cal A}=\|A_1\,|\dots|\, A_q\|,
\qquad A_k=[a_{ijk}]_{ij},
\end{equation*}
(that is, by the list of its
\emph{horizontal slices}).

The transfer from ${\cal A}$ to
${\cal B}$ given by \eqref{1.1}
can be realized in two steps: by
the \emph{simultaneous
equivalence transformation} with
the horizontal slices
\begin{equation}\label{1.4}
\|C_1,\,|\dots|\,Cq\|:=
\|R^TA_1S\,|\dots|\,R^TA_qS\|,
\end{equation}
and then by the nonsingular
linear substitution
\begin{equation}\label{1.5}
B_1=C_1t_{11}+\dots+C_qt_{q1},\
\dots,\
B_q=C_1t_{1q}+\dots+C_qt_{qq},
\end{equation}
where $R$, $S$, and $T$ are the
matrices \eqref{dno}. The last
transformation can be made by
\emph{elementary operations} on
the set $\{C_1,\dots,Cq\}$ of
horizontal slices: interchange
any two slices, multiply one
slice by a non-zero scalar, and
add a scalar multiple of one
slice to another one. This
implies the following lemma.

\begin{lemma}\label{lem0}
Two spatial matrices are
equivalent if and only if one
can be transformed to the other
by a sequence of
\begin{itemize}
  \item[\rm(i)]
simultaneous equivalence
transformations with all
horizontal slices, and
  \item[\rm(ii)]
elementary operations on the set
of horizontal
slices.\hfill$\Box$
\end{itemize}
\end{lemma}

We denote the $m$-by-$n$ zero
matrix by $0_{mn}$. The numbers
$m$ and $n$ may be zero: the
matrices $0_{m0}$ and $0_{0n}$
represent the linear mappings
$0\to {\mathbb F}^m$ and
${\mathbb F}^n\to 0$. For every
$p\times q$ matrix $M_{pq}$ we
have \[ M_{pq}\oplus
0_{m0}=\begin{bmatrix} M_{pq} \\
0_{mq}
\end{bmatrix},\qquad
M_{pq}\oplus
0_{0n}=\begin{bmatrix}
   M_{pq} & 0_{pn}
\end{bmatrix}.
\]

For each natural number $r$, we
define the $(r-1)\times r$
matrices
\begin{equation}\label{nfq}
F_r:=\begin{bmatrix}
 1&0&&0 \\
 &\ddots&\ddots&\\
 0&&1&0
 \end{bmatrix},\qquad
G_r:= \begin{bmatrix}
 0&1&&0 \\
 &\ddots&\ddots&\\
 0&&0&1
\end{bmatrix}.
\end{equation}
For each polynomial
\[
\chi(x)=x^l-u_1 x^{l-1}-\dots-
u_l\in\mathbb F[x],\qquad l\ge
1,
\]
we define the $l\times l$ matrix
\begin{equation}\label{dkr}
\Phi_{\chi}:=\begin{bmatrix}
0&&0&u_l\\1&\ddots&&\vdots\\&\ddots&0&u_2\\
0&&1& u_1
\end{bmatrix},
\end{equation}
whose characteristic polynomial
is $\chi(x)$.

We also define the direct sum of
matrix pairs:
\[
(A,B)\oplus(A',B'):=(A\oplus A',
B\oplus B').
\]

The next theorem will be proved
in Section \ref{sec2}, it
extends Theorem 4.4 of
\cite{[1]} dealing with spatial
matrices over an algebraically
closed field.

\begin{theorem} \label{t1}
Over any field $\mathbb F$,
every $m\times n\times 2$ matrix
${\cal A}=\|A_1| A_2\|$, in
which $\min(m,n)$ is less than
or equal to the number of
elements of\/ $\mathbb F$, is
equivalent to some ${\cal
B}=\|B_1| B_2\|$, in which
\begin{equation}\label{4.1}
(B_1,B_2)=\bigoplus_{i=1}^{p_1}(F_{r_i},G_{r_i})\oplus
\bigoplus_{j=1}^{p_2}
(F_{s_j}^T,G_{s_j}^T)\oplus
\bigoplus_{k=1}^q
(I_{l_k},\Phi_{\chi_k}),
\end{equation}
$p_1,p_2,q$ are nonnegative
integers, all $r_i,s_j,l_k$ are
natural numbers, and each
polynomial $\chi_k$ has degree
$l_k$ and is a power of an
irreducible polynomial. This sum
is determined by ${\cal A}$
uniquely, up to permutation of
summands and up to simultaneous
replacement of all
$\Phi_{\chi_k}$ by
$\Phi_{\eta_k}$ with
\begin{equation}\label{4.1aa}
{\eta}_k(x):=\varepsilon_k(d-x
b)^{l_k}\chi_k\left(\frac{x
a-c}{d-x b}\right),
\end{equation}
where
\begin{itemize}
  \item
$a,b,c,d$ are arbitrary elements
of\/ $\mathbb F$ satisfying
$ad-bc\ne 0$ and
\begin{equation}\label{hdo}
a+b\lambda_k\ne 0 \qquad\text{if
\
$\chi_k(x)=(x-\lambda_k)^{l_k}$},
\end{equation}
  \item
each $ \varepsilon_k$ is a
nonzero element of\/ $\mathbb F$
that makes the coefficient of
the highest order term of
${\eta}_k(x)$ equalling\/ $1$
$($the characteristic polynomial
${\eta}_k(x)$ must be
\emph{monic}$)$.
\end{itemize}
\end{theorem}

Let ${\cal A}=[a_{ijk}]_{i=1}
^m{}_{j=1}^n{}_{k=1}^q$ be a
spatial matrix. Consider the
sets
\begin{equation}\label{mdo}
{\cal S} = \{A_1,\dots,
A_q\},\quad \Tilde{\cal S} =
\{\Tilde{A}_1,\dots,
\Tilde{A}_n\},\quad
\Tilde{\Tilde{{\cal S}}} =
\{\Tilde{\Tilde{A}}_1,\dots,
\Tilde{\Tilde{A}}_m\}
\end{equation}
of its $m\times n$, $m\times q$,
and $n\times q$ submatrices
\[
A_k:= [a_{ij k}]_{ij},\qquad
\Tilde{A}_j:=[a_{ijk}]_{ik},\qquad
\Tilde{\Tilde{A}}_i:=
[a_{ijk}]_{jk}.
\]
We say that $\cal A$ is
\emph{regular} if each of the
sets \eqref{mdo} is linearly
independent.

Suppose $\cal A$ is non-regular
and let $q'$, $n'$, $m'$ be the
ranks of the sets \eqref{mdo}.
Make the first $q'$ matrices in
${\cal S}$ linearly independent
and the others zero by
elementary operations on the set
${\cal S}$. Reduce the ``new''
$\Tilde{\cal S}$ and then the
``new'' $\Tilde{\Tilde{\cal S}}$
in the same way. We obtain a
spatial matrix ${\cal B} =[b_{ij
k}]$, whose $m'\times n'\times
q'$ submatrix
\[
{\cal B}' =
[b_{ijk}]_{i=1}^{m'}{}^{n'}_{j=1}{}
^{q'}_{k=1}
\]
is regular, and whose entries
outside of ${\cal B}'$ are zero;
${\cal B}'$ is called a
\emph{regular part} of ${\cal
A}$. Two spatial matrices of the
same size are equivalent if and
only if their regular parts are
equivalent \cite[Lemma
4.7]{[1]}. Hence, it suffices to
give canonical forms of regular
spatial matrices. The following
theorem will be proved in
Section \ref{sec3}.

\begin{theorem}\label{t2}
Over any field\/ $\mathbb F$,
each regular $m\times n\times q$
matrix $\cal A$ with $n\le 2$
and $q\le 2$ is equivalent to
one of the spatial matrices:
\begin{equation}\label{a9}
\begin{Vmatrix}
  \,1\,
\end{Vmatrix}
\qquad(1\times 1\times 1),
\end{equation}
\begin{equation}\label{a10}
\begin{Vmatrix}
 \, 1&0\,\\0&1
\end{Vmatrix}
\qquad(2\times 2\times 1),
\end{equation}
\begin{equation}\label{a10a}
\left|\!\left|\!\begin{array}{c|c}
    1&0\\ 0&1
\end{array}\!\right|\!\right|
\qquad(2\times 1\times 2),
\end{equation}
\begin{equation}\label{a10b}
\left|\!\left|\!\begin{array}{cc|cc}
    1&0&0&1
\end{array}\!\right|\!\right|
\qquad(1\times 2\times 2),
\end{equation}
\begin{equation}\label{a12}
\left|\!\left|\!\begin{array}{cc|cc}
    1&0&0&0 \\ 0&1&0&0\\ 0&0&0&1
\end{array}\!\right|\!\right|
\qquad(3\times 2\times 2),
\end{equation}
\begin{equation}\label{a12a}
\left|\!\left|\!\begin{array}{cc|cc}
    1&0&0&0 \\ 0&1&1&0\\ 0&0&0&1
\end{array}\!\right|\!\right|
\qquad(3\times 2\times 2),
\end{equation}
\begin{equation}\label{a14z}
\left|\!\left|\!\begin{array}{cc|cc}
    1&0&0&0 \\ 0&1&0&0\\
    0&0&1&0\\0&0&0&1
\end{array}\!\right|\!\right|
\qquad(4\times 2\times 2),
\end{equation}
\begin{equation}\label{a11}
{\cal A}(v)
:=\left|\!\left|\!\begin{array}{cc|cc}
    1&0&0& v\\ 0&1&1&0
\end{array}\!\right|\!\right|
\qquad(v\in\mathbb F,\ 2\times
2\times 2),
\end{equation}
\begin{equation}\label{a11b}
{\cal
B}(v):=\left|\!\left|\!\begin{array}{cc|cc}
    1&0&0& v\\ 0&1&1&1
\end{array}\!\right|\!\right|
\quad(\cha \mathbb F= 2;\
v\in\mathbb F,\ 2\times 2\times
2).
\end{equation}

These spatial matrices are
pairwise inequivalent except for
the following cases:
\begin{itemize}
  \item
If $\cha \mathbb F\ne 2$, then
${\cal A}(v)$ is equivalent to
each ${\cal A}(v')$ with
\begin{equation}\label{mdl1}
v'=vz,\qquad 0\ne z\in\mathbb
F^2:=\{a^2\,|\,a\in\mathbb F\}.
\end{equation}

  \item
If $\cha \mathbb F= 2$, then
${\cal A}(v)$ is equivalent to
each ${\cal A}(v')$ with
\begin{equation}\label{md2}
v'=\frac{\alpha v+\beta }{\gamma
v+\delta}\,,\qquad
\begin{matrix}
\alpha,\beta,\gamma,\delta
\in{\mathbb F}^2,\\
\alpha\delta+\beta\gamma\ne
0,\quad \gamma v+\delta\ne 0,
\end{matrix}
\end{equation}
and ${\cal B}(v)$ is equivalent
to each ${\cal B}(v')$ with
\begin{equation}\label{mspu}
v'=v+\beta+\beta^2,\qquad
\beta\in\mathbb F.
\end{equation}
\end{itemize}

In particular, if\/ $\mathbb F$
is algebraically closed, then
each regular $m\times n\times q$
matrix $\cal A$ with $n\le 2$
and $q\le 2$ is equivalent to
exactly one of the following
spatial matrices:
\eqref{a9}--\eqref{a14z}, ${\cal
A}(0)$, and
\begin{equation}\label{a11be}
\left|\!\left|\!\begin{array}{cc|cc}
    1&0&0& 0\\ 0&0&0&1
\end{array}\!\right|\!\right|
\qquad(2\times 2\times 2).
\end{equation}
\end{theorem}

\section{Proof of Theorem
\ref{t1}}\label{sec2}

We say that two pairs of
matrices of the same size are
\emph{equivalent} if the
matrices of the first pair are
simultaneously equivalent to the
matrices of the second pair.

\begin{lemma}\label{lem1}
Let $(I_l,\Phi_{\chi})$ and
$(I_l,\Phi_{\eta})$ be two
matrix pairs given by arbitrary
monic polynomials ${\chi}$ and
${\eta}$ of degree $l$. Let
\[T:=\begin{bmatrix}
a&c\\b&d
\end{bmatrix},\qquad ad-bc\ne 0,\]
be a nonsingular matrix.

{\rm(a)} If the pair
\begin{equation}\label{ageh}
(aI_{l}+b\Phi_{\chi},
cI_{l}+d\Phi_{\chi})
\end{equation}
is equivalent to
$(I_l,\Phi_{\eta})$, then
\begin{equation}\label{4.1aal}
{\eta}(x)=\varepsilon(d-x
b)^{l}\chi_k\left(\frac{x
a-c}{d-x b}\right)
\end{equation}
for some $\varepsilon\in\mathbb
F$.

{\rm(b)} If \eqref{4.1aal} holds
then the characteristic
polynomials of
\begin{equation}\label{dpx}
(cI_{l}+d\Phi_{\chi})\cdot
(aI_{l}+b\Phi_{\chi})^{-1}
\end{equation}
and $\Phi_{\eta}$ are equal.
\end{lemma}

\begin{proof}
(a) Since the pair \eqref{ageh}
is equivalent to
$(I_l,\Phi_{\eta})$,
$aI_{l}+b\Phi_{\chi}$ is
nonsingular, and so the pair
\eqref{ageh} is equivalent to
\begin{equation}\label{4.4}
\left(I_{l},\
(cI_{l}+d\Phi_{\chi})\cdot
(aI_{l}+b\Phi_{\chi})^{-1}\right).
\end{equation}
Hence \eqref{dpx} is similar to
$\Phi_{\eta}$ and their
characteristic polynomials are
equal:
\begin{align} \nonumber
\eta(x)&=\det \left[xI_l-
(cI_{l}+d\Phi_{\chi})\cdot
(aI_{l}+b\Phi_{\chi})^{-1}\right]
  \\ \nonumber
&=\det
\left[\left[x(aI_{l}+b\Phi_{\chi})-
(cI_{l}+d\Phi_{\chi})\right]\cdot
(aI_{l}+b\Phi_{\chi})^{-1}\right]
   \\ \nonumber
&=\det
\left[(xa-c)I_{l}-(d-xb)\Phi_{\chi}\right]
\cdot\det
(aI_{l}+b\Phi_{\chi})^{-1}
   \\ \nonumber
&=(d-xb)^l\det
\left(\frac{xa-c}{d-xb}I_{l}-
\Phi_{\chi}\right) \cdot\det
(aI_{l}+b\Phi_{\chi})^{-1}
   \\
&=(d-xb)^l\chi
\left(\frac{xa-c}{d-xb}\right)
\cdot\det
(aI_{l}+b\Phi_{\chi})^{-1}.
\label{nsf}
\end{align}
This proves \eqref{4.1aal}.

(b) This statement follows from
\eqref{nsf}.
\end{proof}

Recall \cite{wan} that each
square matrix $A$ over an
arbitrary field $\mathbb F$ is
similar to a matrix of the form
$\Phi=\Phi_{\chi_1}\oplus\dots\oplus
\Phi_{\chi_q},$
where $\chi_1,\dots,\chi_q$ are
powers of an irreducible
polynomials and $\Phi_{\chi_k}$
are defined in \eqref{dkr}. The
matrix $\Phi$ is called the
\emph{Frobenius canonical form}
of $A$ and is determined by $A$
uniquely up to permutations of
summands.

Each pair $(A_1,A_2)$ of
matrices of the same size is
equivalent to a pair of the form
\begin{multline}\label{4.2}
(B_1,B_2)=\bigoplus_{i=1}^{p_1}
(F_{r_i},G_{r_i})\oplus
\bigoplus_{j=1}^{p_2}
(F_{s_j}^T,G_{s_j}^T)\\ \oplus
\bigoplus_{k=1}^{q_1}
(I_{l_k},\Phi_{\chi_k})\oplus
\bigoplus_{k=q_1+1}^{q}
(J_{l_k}(0),I_{l_k}),
\end{multline}
where $p_1,p_2,q_1,q_2$ are
nonnegative integers, $F_r$ and
$G_r$ are defined in
\eqref{nfq}, each polynomial
$\chi_k$ has degree $l_k$ and is
a power of an irreducible
polynomial, and
\[
J_l(\lambda) =
\begin{bmatrix}
  \lambda&&&
0\\
  1&\lambda&&\\
  &\ddots&\ddots&\\
0&&1&\lambda
\end{bmatrix} \quad
\text{($l$-by-$l$)}.
\]
The pair \eqref{4.2} is
determined by $(A_1,A_2)$
uniquely up to permutation of
summands and is called the
\emph{Kronecker canonical form}
of $(A_1,A_2)$ (see, for
example, \cite[Section
1.8]{gab+roi}).

\begin{proof}[Proof of Theorem
\ref{t1}] \emph{Step 1}. Let
$(A_1,A_2)$ be a pair of
matrices of the same size and
let \eqref{4.2} be its Kronecker
canonical form. In this step, we
prove that for each nonsingular
matrix
\[
T=\begin{bmatrix} a&c\\b&d
\end{bmatrix}\in\mathbb
F^{2\times 2},\qquad ad-bc\ne 0,
\]
the Kronecker canonical form of
the pair
\begin{equation}\label{4.3}
(C_1,C_2)=(aB_1+bB_2,\,cB_1+dB_2)
\end{equation}
has the same number $p_1+p_2+q$
of direct summands as
\eqref{4.2} and, after a
suitable permutation of its
summands, it has the same first
$p_1+p_2$ summands as
\eqref{4.2} and the same sizes
$l_1\times l_1,\dots,l_q\times
l_q$ of the remaining $q$
summands as \eqref{4.2}.

A matrix pair is
\emph{decomposable} if it is
equivalent to a direct sum of
pairs of smaller sizes. All
direct summands in \eqref{4.2}
are indecomposable. The
transformation \eqref{4.3} takes
them into indecomposable matrix
pairs. Indeed, if it takes a
summand $\cal P$ into a
decomposable $\cal R$, then the
inverse transformation (given by
the matrix $T^{-1}$) takes $\cal
R$ into a decomposable one,
which is equivalent to $\cal P$,
contrary to the
indecomposability of all direct
summands of \eqref{4.2}.

All indecomposable pairs of
$(r-1)\times r$ or $r\times
(r-1)$ matrices are equivalent
to $(F_r,G_r)$ or, respectively,
$(F_r^T,G_r^T)$. Hence, though
transformations \eqref{4.3} may
spoil the direct summands
$(F_{r_i},G_{r_i})$ and
$(F_{s_j}^T,G_{s_j}^T)$ in
\eqref{4.2}, but they are
restored by equivalence
transformations.
\medskip

\noindent\emph{Step 2}.  Suppose
${\cal A}=\|A_1| A_2\|$
satisfies the hypotheses of
Theorem \ref{t1}. In this step,
we reduce ${\cal A}$ by
equivalence transformations to
some ${\cal B}=\|B_1| B_2\|$
with $(B_1,B_2)$ of the form
\eqref{4.1}.

From the start, we reduce
$(A_1,A_2)$ to the form
\eqref{4.2}.

Thereupon in the case $q_1<q$ we
reduce the pair \eqref{4.2} to a
pair of the form \eqref{4.1}
(with other $\chi_1,\dots,
\chi_{q_1}$) as follows. The
transformation \eqref{4.3} with
\eqref{4.2} given by
\[
T=\begin{bmatrix}
  1&0\\b&1
\end{bmatrix}, \qquad b\ne 0,
\]
takes the direct sum of the last
$q$ summands into
\begin{equation}\label{jdst}
\bigoplus_{k=1}^{q_1} (I_{l_k}+b
\Phi_{\chi_k},\Phi_{\chi_k})\oplus
\bigoplus_{k=q_1+1}^{q}
(J_{l_k}(b),I_{l_k}).
\end{equation}

If some $I_{l_k}+b
\Phi_{\chi_k}$ is singular, then
$\chi_k(x)=(x-b^{-1})^{l_k}$.
Indeed, $0$ is an eigenvalue of
$I_{l_k}+b \Phi_{\chi_k}$, hence
$I_{l_k}+b \Phi_{\chi_k}$ has an
eigenvalue in $\mathbb F$, and
so $\Phi_{\chi_k}$ is similar to
a Jordan block. Further, this
Jordan block must be
$J_{l_k}(-b^{-1})$.

In view of the hypotheses of
Theorem \ref{t1}, $\min(m,n)$ is
less than or equal to the number
of elements of\/ $\mathbb F$.
Since $q_1<q\le \min(m,n)$, the
number $q_1$ of the summands
$(I_{l_k},\Phi_{\chi_k})$ in
\eqref{4.2} is less than or
equal to the number of nonzero
elements of\/ $\mathbb F$.

First suppose that one of these
summands is
$(I_{l_k},J_{l_k}(0))$. Then
there exists a nonzero $b\in
\mathbb F$ such that
$\chi_k(x)\ne (x-b^{-1})^{l_k}$
for all $k\le q_1$, this means
that all $I_{l_k}+b
\Phi_{\chi_k}$ are nonsingular.
We take such $b$ and reduce
\eqref{jdst} to the form
\begin{equation}\label{jts}
(I_{l_1},\Phi_{\eta_1})\oplus\dots
\oplus (I_{l_q},\Phi_{\eta_q})
\end{equation}
by equivalence transformations.

Now suppose that there are no
summands $(I_{l_k},J_{l_k}(0))$.
Then the second matrix in each
of the last $q$ summands of
\eqref{4.2} is nonsingular. We
interchange the matrices $B_1$
and $B_2$ in the pair
\eqref{4.2} and reduce its last
$q$ summands to the form
\eqref{jts}.
\medskip

\noindent\emph{Step 3}. Suppose
${\cal A}=\|A_1| A_2\|$ is
equivalent both to ${\cal
B}=\|B_1| B_2\|$ with
$(B_1,B_2)$ of the form
\eqref{4.1} and to another
${\cal B}'=\|B'_1| B'_2\|$ with
\begin{equation}\label{4.1dh}
(B'_1,B'_2)=\bigoplus_{i=1}^{p'_1}
(F_{r'_i},G_{r'_i})\oplus
\bigoplus_{j=1}^{p'_2}
(F_{s'_j}^T,G_{s'_j}^T)\oplus
\bigoplus_{k=1}^{q'}
(I_{l'_k},\Phi_{\eta_k}).
\end{equation}
Let us prove that \eqref{4.1dh}
coincides, after a suitable
permutation of its summands,
with \eqref{4.1} except for
$\chi_k$ and $\eta_k$, and that
\eqref{4.1aa} is fulfilled.

Since ${\cal B}$ and  ${\cal
B}'$ are equivalent, by Lemma
\ref{lem0} $(B'_1,B'_2)$ is the
Kronecker canonical form of some
pair $(C_1,C_2)$ of the form
\eqref{4.3} with $ad-bc\ne 0$.
In view of Step 1,
$p'_1,p'_2,q'$ and all
$r_i',s_j',l_k'$ coincide with
$p_1,p_2,q$ and all
$r_i,s_j,l_k$ after a suitable
permutation of the summands of
\eqref{4.1dh}.

The transformation \eqref{4.3}
converts each summand
$(I_{l_k},\Phi_{\chi_k})$ of
\eqref{4.1} to the matrix pair
\begin{equation}\label{age}
(aI_{l_k}+b\Phi_{\chi_k},
cI_{l_k}+d\Phi_{\chi_k}),
\end{equation}
which is equivalent to
$(I_{l_k},\Phi_{\eta_k})$. The
matrix $aI_{l_k}+b\Phi_{\chi_k}$
is nonsingular; this means that
if $\Phi_{\chi_k}$ is similar to
some Jordan block
$J_{l_k}(\lambda_k)$, then
$a+b\lambda_k\ne 0$; we have the
condition \eqref{hdo}. Due to
Lemma \ref{lem1}(a), $\eta_k(x)$
is represented in the form
\eqref{4.1aa}.

Conversely, let $(B_1,B_2)$ of
the form \eqref{4.1} and
\eqref{4.1dh} coincide with
except for $\chi_k$ and $\eta_k$
that satisfy \eqref{4.1aa}. By
Lemma \ref{lem1}(b), the
characteristic polynomials of
the matrices
\begin{equation}\label{vsp}
(cI_{l_k}+d\Phi_{\chi_k})\cdot
(aI_{l_k}+b\Phi_{\chi_k})^{-1}
\end{equation}
and $\Phi_{\eta_k}$ are equal
for each $k$. Since
$(I_{l_k},\Phi_{\chi_k})$ is
indecomposable, in view of Step
1 the matrix pair \eqref{age} is
indecomposable too, hence the
matrix \eqref{vsp} is
indecomposable with respect to
similarity and its Frobenius
canonical form is
$\Phi_{\eta_k}$. Therefore, each
$\|I_{l_k}|\Phi_{\chi_k}\|$ is
equivalent to
$\|I_{l_k}|\Phi_{\eta_k}\|$, and
so $\cal A$ is equivalent to
$\cal B$.
\end{proof}

\section{Proof of Theorem
\ref{t2}}\label{sec3}

\begin{lemma}\label{lem2}
A spatial matrix
\begin{equation}\label{msp}
{\cal
D}(u,v):=\left|\!\left|\!\begin{array}{cc|cc}
    1&0&0& v\\ 0&1&1&u
\end{array}\!\right|\!\right|,\qquad
u,v\in\mathbb F,
\end{equation}
is equivalent to ${\cal
D}(u',v')$  if and only if there
exist $a,b,c,d\in \mathbb F$
such that
\begin{equation}\label{ngt}
ad-bc\ne 0,\qquad
a^2+uab-vb^2\ne 0
\end{equation}
and
\begin{equation}\label{vso}
u'=\frac{2ac+uad+ucb
-2vbd}{a^2+uab-vb^2}, \qquad
v'=\frac{-c^2-ucd+vd^2}
{a^2+uab-vb^2}.
\end{equation}
\end{lemma}

\begin{proof}
Notice that \[{\cal
D}(u,v)=\|I_{2}|\Phi_{\chi}\|,\qquad
{\cal
D}(u',v')=\|I_{2}|\Phi_{\eta}\|,\]
where
\[
\chi(x):=x^2-ux-v,\qquad
\eta(x):=x^2-u'x-v'.
\]
\medskip

 ``$\Longrightarrow$''. Let
$\|I_{2}|\Phi_{\chi}\|$ and
$\|I_{2}|\Phi_{\eta}\|$ be
equivalent. By Lemma \ref{lem0},
there exists a nonsingular
matrix \[\begin{bmatrix}
a&c\\b&d
\end{bmatrix},\qquad ad-bc\ne 0,\]
such that the pairs
\begin{equation}\label{des}
(aI_2+b\Phi_{\chi},
cI_2+d\Phi_{\chi}),\qquad
(I_2,\Phi_{\eta})
\end{equation}
are equivalent. Then
$aI_2+b\Phi_{\chi}$ is
nonsingular; i.e., \[\det(
aI_2+b\Phi_{\chi})=a^2+uab-vb^2\ne
0.\] By Lemma \ref{lem1}(a),
$\eta(x)$ satisfies
\eqref{4.1aal}, this means that
for some nonzero $\varepsilon$
\begin{multline}\label{smf}
\eta(x)=\varepsilon\left[(xa-c)^2-
u(xa-c)(d-xb)-v(d-xb)^2\right]\\
     =\varepsilon\left[
x^2(a^2+uab-vb^2)+x(-2ac-uad-ucb
+2vbd)\right.\\
+\left.(c^2+ucd-vd^2)\right]
=x^2-u'x-v'.
\end{multline}
Therefore,
$\varepsilon=(a^2+uab-vb^2)^{-1}$
and the conditions \eqref{ngt}
and \eqref{vso} hold true.
\medskip

 ``$\Longleftarrow$''.
Conversely, let \eqref{ngt} and
\eqref{vso} hold. Then
\eqref{smf} is fulfilled and we
have \eqref{4.1aal}. By Lemma
\ref{lem1}(b), the
characteristic polynomials of
\begin{equation}\label{dpf}
(cI_2+d\Phi_{\chi})\cdot
(aI_2+b\Phi_{\chi})^{-1}
\end{equation}
and $\Phi_{\eta}$ are equal.
Since \eqref{dpf} is 2-by-2,
this implies that its Frobenius
canonical form is either
$\Phi_{\eta}$, or a direct sum
of two 1-by-1 Frobenius blocks
$\lambda I_1\oplus \mu I_1$ for
some $\lambda,\mu\in\mathbb F$.

In the last case,
$\eta(x)=(x-\lambda)(x-\mu)$.
But $\eta(x)$ is a power of an
irreducible polynomial. Hence,
$\lambda=\mu$ and \eqref{dpf} is
$\lambda I_2$. We get
consecutively
\begin{gather*}
 cI_{2}+d\Phi_{\chi}=\lambda
(aI_{2}+b\Phi_{\chi}),\qquad
(c-\lambda a)I_2=(\lambda
b-d)\Phi_{\chi},\\
c-\lambda a=\lambda b-d=0,\qquad
(c,d)=\lambda(a,b),
\end{gather*}
contrary to $ad-bc=0$.

Therefore, \eqref{dpf} is
similar to $\Phi_{\eta}$, the
pairs \eqref{des} are
equivalent, and so
$\|I_{2}|\Phi_{\chi}\|$ is
equivalent to
$\|I_{2}|\Phi_{\eta}\|$.
\end{proof}

\begin{proof}[Proof of Theorem
\ref{t2}] Let $\cal A$ be a
regular $m\times n\times q$
matrix with $n\le 2$ and $q\le
2$.
\medskip

\noindent\emph{Step 1.} Let us
prove that ${\cal A}$ is
equivalent to at least one of
the spatial matrices
\eqref{a9}--\eqref{a11b}. This
is clear if ${\cal A}$ is
$m\times n\times 1$ with $n\le
2$: indeed, since ${\cal
A}=\|A\|$ is regular, it reduces
by elementary transformations
\eqref{1.4} to \eqref{a9} or
\eqref{a10}.

So we suppose that ${\cal A}$ is
$m\times n\times 2$ with $n\le
2$. By Theorem \ref{t1}, ${\cal
A}$ is equivalent to some ${\cal
B}=\|B_1| B_2\|$ with $(B_1,
B_2)$ of the form \eqref{4.1}.
Since ${\cal A}$ is regular,
\eqref{4.1} does not have the
summands $(F_{1},G_{1})$ and
$(F_{1}^T,G_{1}^T)$. If $m=1$ or
$n=1$, then $(B_1, B_2)$ is
$(F_{2},G_{2})$ or
$(F_{2}^T,G_{2}^T)$, we have
\eqref{a10b} or \eqref{a10a}.

It remains to consider ${\cal
A}$ of size $m\times 2\times 2$
with $m\ge 2$. Then $(B_1, B_2)$
is one of the pairs:
\begin{gather}\label{ndiu}
(F_{3}^T,G_{3}^T),\quad
(F_{2}^T,G_{2}^T)\oplus
(F_{2}^T,G_{2}^T),\quad
(F_{2}^T,G_{2}^T)\oplus
(I_1,J_1(\lambda)),
    \\ \label{ndiu1}
(I_1,J_1(\lambda))\oplus
(I_1,J_1(\mu)), \quad
(I_{2},\Phi_{\chi}).
\end{gather}
The first and the second pairs
give \eqref{a12a} and
\eqref{a14z}. In the third pair
we take $\lambda=0$ (because
$\|1|\lambda\|$ and $\|1|0\|$
are equivalent) and obtain
\eqref{a12}. In the fourth pair,
$\lambda\ne \mu$ since $\cal A$
is regular, and so it is
equivalent to
$(I_{2},\Phi_{\chi})$ with
$\chi(x)=(x-\lambda)(x-\mu)$.

Hence, the spatial matrices that
are given by \eqref{ndiu1} are
equivalent to ${\cal D}(u,v)$
defined in \eqref{msp}.

If $\cha \mathbb F\ne 2$, then
each ${\cal D}(u,v)$ is
equivalent to ${\cal D}(0,v')$
for some $v'$ due to Lemma
\ref{lem2}: substituting
\[(a,b,c,d):=(1,0,-u/2,1)\]
in \eqref{ngt} and \eqref{vso},
we obtain $u'=0$. This gives
\eqref{a11}.

If $\cha \mathbb F=2$, then each
${\cal D}(u,v)$ is equivalent to
${\cal D}(0,v')$ or ${\cal
D}(1,v')$: for each $u\ne 0$ we
get $u'=1$ putting
\[(a,b,c,d):=(1,0,0,u^{-1})\]
in \eqref{ngt} and \eqref{vso}.
This gives \eqref{a11} and
\eqref{a11b}.
\medskip

\noindent \emph{Step 2.} Let us
prove that $\cal A$ is
equivalent to exactly one of the
spatial matrices
\eqref{a9}--\eqref{a11b} up to
replacements
\eqref{mdl1}--\eqref{mspu}.

Let two distinct spatial
matrices among
\eqref{a9}--\eqref{a11b} be
equivalent. Then they have the
same size, and so they are
$3\times 2\times 2$ or $2\times
2\times 2$. The spatial matrices
\eqref{a12} and \eqref{a12a} are
inequivalent in view of Theorem
\ref{t1} since the corresponding
decompositions \eqref{4.1} are
$(F_{2}^T,G_{2}^T)\oplus
(I_1,J_1(0))$ and
$(F_{3}^T,G_{3}^T)$. Hence, they
are \eqref{a11} or \eqref{a11b}.

Let $\cha \mathbb F\ne 2$, and
let ${\cal A}(v)$ be equivalent
to ${\cal A}(v')$. By Lemma
\ref{lem2} there exist $a,b,c,d$
satisfying \eqref{ngt} and
\eqref{vso} with $u=u'=0$. Then
the equalities \eqref{vso}
ensure
\[
ac-vbd=0,\qquad
v'=\frac{-c^2+vd^2}{a^2-vb^2},
\]
and so
\begin{align*}
v'(a^2-vb^2)^2&=(a^2-vb^2)(-c^2+vd^2)
=-a^2c^2+a^2vd^2+vb^2c^2-v^2b^2d^2\\
&=-(ac-vbd)^2+v(ad-bc)^2=v(ad-bc)^2.
\end{align*}
We have \eqref{mdl1} with
\[
z=\left(\frac{a^2-vb^2}{ad-bc}\right)^2.
\]
Conversely, if \eqref{mdl1}
holds, then ${\cal A}(v)$ is
equivalent to ${\cal A}(v')$ due
to Lemma \ref{lem2} since the
conditions \eqref{ngt} and
\eqref{vso} are fulfilled with
\[
u=u'=0,\qquad
(a,b,c,d):=(1,0,0,z^{-1/2}) .
\]

Let $\cha \mathbb F=2$. If
${\cal D}(0,v)$ is equivalent to
${\cal D}(u',v')$, then by
\eqref{vso} $u'=0$. Hence ${\cal
A}(v)$ and ${\cal B}(v')$
(defined in \eqref{a11} and
\eqref{a11b}) are inequivalent
for all $v$ and $v'$.

Due to Lemma \ref{lem2}, ${\cal
A}(v)$ and ${\cal A}(v')$ are
equivalent if and only if the
conditions \eqref{ngt} and
\eqref{vso} with $u=u'=0$ hold
for some $a,b,c,d\in \mathbb F$.
The first condition in
\eqref{vso} is the identity,
putting
\[(\alpha,\beta,\gamma,\delta):=
(d^2,c^2,b^2,a^2) \] in the
other conditions gives the
conditions \eqref{md2}.

Let ${\cal B}(v)$ be equivalent
to ${\cal B}(v')$. Then there
exist $a,b,c,d$ such that the
conditions \eqref{ngt} and
\eqref{vso} hold for $u=u'=1$.

We first suppose that $b=0$. The
conditions \eqref{vso} take the
form $1=ad/a^2$ (and so $a=d\ne
0$) and
\[
v'=\frac{c^2+ca+va^2} {a^2}=v+
\frac{c}{a}+\frac{c^2}{a^2};
\]
this gives \eqref{mspu} with
$\beta=c/a.$

Let now $b\ne 0$. Denote
\[\alpha :=a/b,\qquad\gamma:=c/b,\qquad
\delta:=d/b.\] Remembering that
$\cha\mathbb F=0$ and $u=u'=1$,
rewrite \eqref{vso} in the form
\begin{equation*}\label{vsod}
1=\frac{\alpha\delta+\gamma}
{\alpha^2+\alpha+v}\,, \qquad
v'=\frac{\gamma^2+\gamma \delta
+v\delta ^2}{\alpha^2+\alpha+v}.
\end{equation*}
From the first equality
determine
\[
\gamma=v+\alpha +\alpha^2+\alpha
\delta,
\]
substitute it to the second:
\begin{align*}
v'&= \frac{(v+\alpha
+\alpha^2)^2+(\alpha
\delta)^2+(v+\alpha
+\alpha^2)\delta+\alpha \delta^2
+v\delta
^2}{\alpha^2+\alpha+v}\\[1mm] &=
v+\alpha +\alpha^2+\delta+
\delta^2
 =v+(\alpha+
\delta)+(\alpha+ \delta)^2,
\end{align*}
and obtain \eqref{mspu}.

Conversely, let
$v'=v+\beta+\beta^2$ for some
nonzero $\beta \in\mathbb F$.
The conditions \eqref{ngt} and
\eqref{vso} hold for $u=u'=1$
and
\[
(a,b,c,d):=
  \begin{cases}
   (\beta,1,v', 0)& \text{if
$v'\ne 0$}, \\
(1,0,\beta,1) & \text{If
$v'=0$};
  \end{cases}
\]
hence ${\cal B}(v)$ and ${\cal
B}(v')$ are equivalent by Lemma
\ref{lem2}.
\medskip

\noindent \emph{Step 3.} Let
$\mathbb F$ be algebraically
closed. Then $\mathbb
F^2=\mathbb F$. If $\cha \mathbb
F\ne 2$, each ${\cal A}(v)$ is
equivalent to ${\cal A}(0)$ or
${\cal A}(1)$ (if $v\ne 0$, we
put $z=1/v$ in \eqref{mdl1}).
The spatial matrix ${\cal A}(1)$
is equivalent to \eqref{a11be}
since it reduces to
\eqref{a11be} by the following
transformations: add the first
slice $I_2$ to the second,
reduce the second to the form
$J_1(0)\oplus J_1(2)$ by
simultaneous similarity
transformations with the slices,
divide the second by $2$, and
subtract the second slice from
the first:
\[
\arraycolsep=1mm {\cal A}(1)
     \to
\left|\!\left|\!\begin{array}{cc|cc}
    1&0&1&1\\ 0&1&1&1
\end{array}\!\right|\!\right|
    \to
\left|\!\left|\!\begin{array}{cc|cc}
    1&0&0& 0\\ 0&1&0&2
\end{array}\!\right|\!\right|
\to
\left|\!\left|\!\begin{array}{cc|cc}
    1&0&0& 0\\ 0&1&0&1
\end{array}\!\right|\!\right|
\to
\left|\!\left|\!\begin{array}{cc|cc}
    1&0&0& 0\\ 0&0&0&1
\end{array}\!\right|\!\right|.
\]

Suppose $\cha \mathbb F=2$. Then
all ${\cal A}(v)$ are equivalent
to ${\cal A}(0)$ since the
conditions \eqref{md2} hold for
$v'=0$ and $(\alpha, \beta,
\gamma, \delta):=(1,v,0,1)$. All
${\cal B}(v)$ are equivalent to
${\cal B}(0)$ since the equation
\eqref{mspu} with $v'=0$ is
solvable for $\beta$. We reduce
the second slide of ${\cal
B}(0)$ to the form $J_1(0)\oplus
J_1(1)$  by simultaneous
similarity transformations with
the slices, and then subtract
the second slice from the first.
\end{proof}

\end{document}